\documentclass[11pt]{amsart}
\usepackage[dvips]{graphicx}
\usepackage{amsmath}
\usepackage{amssymb}
\usepackage{amscd}
\usepackage{mathrsfs}
\usepackage{color}

\usepackage[all]{xy}%

\DeclareFontEncoding{OT2}{}{}
\DeclareFontSubstitution{OT2}{cmr}{m}{n}
\DeclareFontFamily{OT2}{cmr}{\hyphenchar\font45 }
\DeclareFontShape{OT2}{cmr}{m}{n}{%
   <5><6><7><8><9>gen*wncyr%
   <10><10.95><12><14.4><17.28><20.74><24.88>wncyr10}{}
\DeclareFontShape{OT2}{cmr}{b}{n}{%
   <5><6><7><8><9>gen*wncyb%
   <10><10.95><12><14.4><17.28><20.74><24.88>wncyb10}{}

\DeclareMathAlphabet{\mathcyr}{OT2}{cmr}{m}{n}
\DeclareMathAlphabet{\mathcyb}{OT2}{cmr}{b}{n}
\SetMathAlphabet{\mathcyr}{bold}{OT2}{cmr}{b}{n}
\DeclareMathOperator*{\foo}{\SH}

\theoremstyle{plain}
    \newtheorem{theorem}{Theorem}[section]
\theoremstyle{definition}
    \newtheorem{definition}[theorem]{Definition}
    \newtheorem{proposition}[theorem]{Proposition}
    
    \newtheorem{corollary}[theorem]{Corollary}
    
    \newtheorem{example}[theorem]{Example}
    \newtheorem{remark}[theorem]{Remark}

\newcommand{\bF}{\mathbb{F}}
\newcommand{\bZ}{\mathbb{Z}}
\newcommand{\bQ}{\mathbb{Q}}
\newcommand{\bR}{\mathbb{R}}
\newcommand{\cA}{\mathcal{A}}
\newcommand{\frH}{\mathfrak{H}}
\newcommand{\sh}{\mathbin{\mathcyr{sh}}}
\newcommand{\SH}{\mathbin{\mathcyr{Sh}}}

\newcommand{\bs}{\mathbf{s}}
\newcommand{\bp}{\mathbf{p}}
\newcommand{\h}{\mathrm{h}}

\newcommand{\rt}{\mathrm{rt}}

\renewcommand{\labelenumi}{(\roman{enumi})}

\address{Department of Mathematics, Keio University, 3-14-1 Hiyoshi, Kouhoku-ku,Yokohama 223-8522, Japan}
\email{ono@math.keio.ac.jp}

\title{Finite multiple zeta values associated with 2-colored rooted trees}
\author{Masataka Ono}
\thanks{This research was supported by KLL 2016 Ph.D. Program Research Grant of Keio University.  This research was also supported in part by KAKENHI 26247004, as well as the JSPS Core-to-Core program ``Foundation of a Global Research Cooperative
Center in Mathematics focused on Number Theory and Geometry" and the KiPAS program 2013--2018 of the Faculty of Science and Technology at Keio University. }

\begin{document}

\maketitle
\begin{abstract}
We define finite multiple zeta values (FMZVs) associated with some combinatorial objects, which we call 2-colored rooted trees, and prove that FMZVs associated with 2-colored rooted trees satisfying certain mild assumptions can be written explicitly as $\bZ$-linear combinations of the usual FMZVs. Our result can be regarded as a generalization of Kamano's recent work \cite{K} on finite Mordell-Tornheim multiple zeta values. As an  application, we will give a new proof of the shuffle relation of FMZVs, which was first proved by Kaneko and Zagier.
\end{abstract}

\tableofcontents

\section{Introduction}


In the 1990's, Hoffman \cite{H2} and Zhao \cite{Z} had started independently a theory of mod $p$ multiple harmonic sums. Recently, Kaneko and Zagier introduced a new ``adelic" framework to describe the work of Hoffman and Zhao. That is, they defined the $\bQ$-algebra $\cA:=\left.\left(\prod_p\bF_p\right)\right/\left(\bigoplus_p\bF_p\right)$, where $p$ runs through all rational primes. Thus, an element of $\cA$ is represented by a family $(a_p)_p$ of elements $a_p \in \bF_p$, and two families $(a_p)_p$ and $(b_p)_p$ represent the same element of $\cA$ if and only if $a_p=b_p$ for all but finitely many rational primes $p$.

\begin{definition}[\cite{KZ}]
Finite multiple zeta value is an element of $\cA$ defined by 
$$
\zeta_\cA(k_1, \ldots, k_r):=\left(\sum_{0 < n_1<\cdots<n_r < p}\frac{1}{n^{k_1}_1n^{k_2}_2\cdots n^{k_r}_r}\right)_p.
$$
We call the integers $k_1+\cdots+k_r$ and $r$ the weight and the depth of $\zeta_\cA(k_1, \ldots, k_r)$, respectively.
\end{definition}

In the following, we often denote an element $(a_p)_p$ of $\cA$ simply by $a_p$ omitting $(\ )_p$ if there is no fear of confusion. For example, the above definition is written as 
$$
\zeta_\cA(k_1, \ldots, k_r)=\sum_{0 < n_1<\cdots<n_r < p}\frac{1}{n^{k_1}_1n^{k_2}_2\cdots n^{k_r}_r}. 
$$

There exist many research of FMZVs, for example, \cite{M}, \cite{O}, \cite{SS}, \cite{SW1}, \cite{SW2}. Our research in the present paper is motivated by Kamano's work \cite{K} on the finite Mordell Tornheim multiple zeta values (FMZVs of MT-type):
$$
\zeta^{MT}_\cA(k_1, \ldots, k_r; k_{r+1}):=\sum_{\substack{m_1, \ldots, m_r \geq1 \\ m_1+\cdots+m_r \leq p-1}}
\frac{1}{m^{k_1}_1m^{k_2}_2\cdots m^{k_r}_r(m_1+\cdots+m_r)^{k_{r+1}}} \in \cA,
$$
where $k_1, \ldots, k_r$ are positive integers and $k_{r+1}$ is a non-negative integer. Kamano proved that FMZVs of MT-type can be written explicitly as $\bZ$-linear combinations of the usual (i.e., Kaneko-Zagier's) FMZVs \cite[Theorem 2.1]{K}. Using this result, he obtained non-trivial $\bZ$-linear relations, which we call Kamano's relation, among FMZVs \cite[Theorem 3.2, Proposition 3.4]{K}. 






In the present paper, we define FMZVs associated with 2-colored rooted trees and prove that, with mild assumptions, they can be written explicitly as a $\bZ$-linear combination of the usual FMZVs. The following definition is inspired by Yamamoto's work \cite{Y}.

\begin{definition}\label{2crt}
A 2-colored rooted tree $(T, \rt_T, V_\bullet)$ is a triple consisting of the following data.
\begin{enumerate}

\item $T=(V, E)$ is a tree  (in the graph theoretic sense) such that $\#V(=\#E+1)<\infty$.

\item $\rt_T \in V$ is a vertex, called the root.

\item $V_\bullet$ is a subset of $V$ containing all terminals of $T$. We set $V_\circ:=V\setminus V_\bullet$.
\end{enumerate}
\end{definition}

In the following, for a 2-colored rooted tree $(T, \rt_T, V_\bullet)$ and an edge $e$ of $T$ and $(m_v)\in \bZ^{V_\bullet}_{\geq1}$, we set 
$$
L_e(\rt_T, (m_v)):=\sum_{v \in V_\bullet \text{ s.t. } e \in P(\rt_T, v)}m_v,
$$
where $P(\rt_T, v)$ is the path from $\rt_T$ to $v$.

\begin{definition}\label{FMZV_assoc_2crt}
For a 2-colored rooted tree $X=(T, \rt_T, V_\bullet)$ and a map $k : E \rightarrow \bZ_{\geq0}$, the FMZV $\zeta_\cA(X,k)$ associated with $X$ is defined as an element in $\cA$ by
$$
\zeta_\cA(X, k):=\sum_{\substack{(m_v)\in\bZ^{V_\bullet}_{\geq1} \text{ s.t.} \\ \sum_{v \in V_\bullet}m_v=p}}
\prod_{e \in E}L_e(\rt_T, (m_v))^{-k(e)}
$$
We call $k : E \rightarrow \bZ_{\geq0}$ an index on $X$.
\end{definition}

Our main theorem is the following.

\begin{theorem}\label{mt}
Let $X$ be a 2-colored rooted tree and $k$ an index on $X$. Suppose that $\sum_{e\in P(v,v')}k(e)$ is positive for any $v, v' \in V_\bullet$. Then, the FMZV $\zeta_\cA(X, k)$ can be written explicitly as a $\bZ$-linear combination of Kaneko-Zagier's FMZVs.
\end{theorem}

Since a FMZV of MT-type coincides with $\zeta_\cA(X,k)$ for a specific $(X, k)$ (see Example \ref{eg} (i)), our main result is a generalization of Kamano's result mentioned above.

As a corollary of our main theorem, we give another proof of the shuffle relation among FMZVs, which was first proved by Kaneko and Zagier \cite{KZ} (see Corollary \ref{sh_prod_FMZV}). Therefore our result can be regarded as a simultaneous generalization of both Kamano's relation and the shuffle relation. We should note that this is very surprising since there were no obvious connections between these two classes of relations.

\section{Basic properties of FMZV associated with 2-colored rooted trees}

In this section, we will give two examples of 2-colored rooted trees and FMZVs associated with them.  These examples show that the FMZV associated with a 2-colored rooted tree is a generalization of the usual FMZV and the FMZV of MT-type. Next, we prove the three basic properties of the FMZVs associated with 2-colored rooted trees. The first and second properties are about contracting certain edges of 2-colored rooted trees and the third is about changing the roots of the given 2-colored rooted trees. Using these properties, we define the notion ``harvestable" for a pair consisting of a 2-colored rooted tree and an index on it. The proof of our main theorem will be reduced to the case when the pair is harvestable.

\begin{example}\label{ex_of_2crt}
We use diagrams to indicate 2-colored rooted trees $(T, \rt_T, V_\bullet)$, with symbols $\circ$ and $\bullet$ corresponding to the vertices in $V_\circ$ and $V_\bullet$, respectively. 
\begin{enumerate}

\item Let $X=(T, \rt_T, V_\bullet)$ be a 2-colored rooted tree and $k$ an index on $X$ as follows. 
          $$
          \xybox{ 
          {(0,-8) \ar_{k_1} @{{*}-{*}} (0,-2)},
          {(0,-2) \ar @{.{*}} (0,6)},
          {(0,6) \ar_{k_r} @{-} (0,12)},
          {(0,13)*++{\scriptstyle \blacksquare} \ar @{} (-3,-8)*++{\scriptstyle {\large v_1}} \ar @{} (-3,-2)*++{\scriptstyle v_2} \ar @{} (-3,6)*++{\scriptstyle v_r} \ar @{} (-5,13)*++{\scriptstyle v_{r+1}}},
          }
          $$ 
          Here, $k_i:=k(e_i)$ and $e_i \in E$ is an edge of $T$ and $\rt_T=v_{r+1}$.
          If we set $m_i:=m_{v_i} \;(1 \leq i \leq r+1)$, since $L_{e_i}(\rt_T, (m_v))=m_1+\cdots+m_i \; (1 \leq i \leq r)$, we obtain
          \begin{align*}
          \zeta_\cA(X, k)
          =&\sum_{\substack{m_1, \ldots, m_{r+1} \geq 1 \\ m_1+\cdots+m_{r+1}=p}}(m_1+\cdots+m_r)^{-k_r}\cdots(m_1+m_2)^{-k_2}m^{-k_1}_1\\
          =&\sum_{\substack{m_1, \ldots, m_r \geq 1 \\ m_1+\cdots+m_r \leq p-1}}\frac{1}{m^{k_1}_1(m_1+m_2)^{k_2}\cdots(m_1+\cdots+m_r)^{k_r}}\\
          =&\zeta_\cA(k_1, \ldots, k_r).
          \end{align*}
          Thus, the usual FMZV coincides with the FMZV associated with the above 2-colored rooted tree.

\item Next, consider the following 2-colored rooted tree $X=(T, \rt_T, V_\bullet)$ and the index $k$ on $X$. 
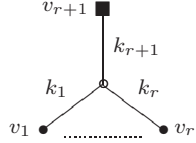
          $$
          \xybox{
          {(-8,-4) \ar^{k_1} @{{*}-o} (0,2)}, 
          {(8,-4) \ar_{k_r} @{{*}-o} (0,2)}, 
          {(0,2) \ar_{k_{r+1}} @{-} (0,12)}, 
          {(-5,-5) \ar @{.} (5,-5)}, 
          {(0,12)*++{\scriptstyle \blacksquare} \ar @{} (-11,-4)*++{\scriptstyle v_1} \ar @{} (11,-4)*++{\scriptstyle v_r} \ar @{}(-5,12)*++{\scriptstyle v_{r+1}}},
          }
          $$
          Assume that $\rt_T=v_{r+1}$ and $k_i \geq 1 (1 \leq i \leq r)$.
          Since $L_{e_i}(\rt_T, (m_v))=m_i \; (1\leq i \leq r)$ and $L_{e_{r+1}}(\rt_T, (m_v))=m_1+\cdots+m_r$, we obtain
          \begin{align*}
          \zeta_\cA(X, k)
          =&\sum_{\substack{m_1, \ldots, m_{r+1} \geq 1 \\ m_1+\cdots+m_{r+1}=p}}m^{-k_1}_1\cdots m^{-k_r}_r(m_1+\cdots+m_r)^{-k_{r+1}}\\
          =&\sum_{\substack{m_1, \ldots, m_r \geq 1 \\ m_1+\cdots+m_r\leq p-1}}\frac{1}{m^{k_1}_1\cdots m^{k_r}_r(m_1+\cdots+m_r)^{k_{r+1}}}\\
          =&\zeta^{MT}_\cA(k_1, \ldots, k_r ; k_{r+1}).
          \end{align*}
          Thus we see that the FMZV of MT-type is a special case of the FMZV associated with the 2-colored rooted trees.

\end{enumerate}
\end{example}


\begin{proposition}\label{contracting1}
Let $X=(T, \rt_T, V_\bullet)$ be a 2-colored rooted tree and $k$ be an index on $X$. Assume that there exists an edge $e \in E$ satisfying that one of the end points of $e$ is in $V_\circ$ and $k(e)=0$. Consider the quotient tree $\overline{T}:=T/e=(\overline{V}, \overline{E})$ obtained by contracting $e$. Let $\overline{V}_\bullet$ and $\overline{\rt_T} \in \overline{V}_\bullet$ be the images of $V_\bullet$ and $\rt_T$ in $\overline{T}$, and take $\rt_{\overline{T}}:=\overline{\rt_T}$. These situations can be written as the following figures.
$$
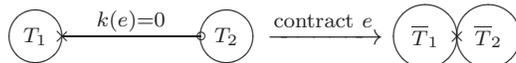
\xybox{
{(-12.7,0)*++[o][F]{\scriptstyle T_1}="1"},
{(12.7,0)*++[o][F]{\scriptstyle T_2}="2'"},
{(-9.21,0) \ar @{{x}-{o}} (9.21,0)},
{(0,2)*++{\scriptstyle k(e)=0}},
}
\xrightarrow{\text{contract $e$}}
\xybox{
{(-4.2,0)*++[o][F]{\scriptstyle \overline{T}_1}="1"},
{(4.2,0)*++[o][F]{\scriptstyle \overline{T}_2}="2'"},
{(0,0) \ar @{{x}} (0,0)},
}
$$
Here, $\times$ is $\circ$ or $\bullet$.
Let $\overline{k}$ be an index on $\overline{X}:=(\overline{T}, \rt_{\overline{T}}, \overline{V}_\bullet)$ defined by $\overline{k}(\overline{e'}):= k(e')$ if $\overline{e'} \in \overline{E}$ is the image of $e'\in E$. 
Then we have
$$
\zeta_\cA(X, k)=\zeta_\cA(\overline{X}, \overline{k}).
$$
\end{proposition}

\begin{proof}
Since $k(e)=0$, we have
\begin{align*}
   \zeta_\cA(X, k)
=&\sum_{\substack{(m_v) \in \bZ^{V_\bullet}_{\geq1} \text{ s.t.}\\ \sum_{v \in V_\bullet}m_v=p}}L_e(\rt_T, (m_v))^{-k(e)}
     \prod_{e'\in E\setminus\{e\}}L_{e'}(\rt_T, (m_v))^{-k(e')}\\
=&\sum_{\substack{(m_{\overline{v}}) \in \bZ^{\overline{V}_\bullet}_{\geq1} \text{ s.t.}\\ \sum_{\overline{v} \in \overline{V}_\bullet}m_{\overline{v}}=p}}
     \prod_{\overline{e'}\in \overline{E}}L_{\overline{e'}}(\rt_{\overline{T}}, (m_{\overline{v}}))^{-\overline{k}(\overline{e'})}=\zeta_\cA(\overline{X}, \overline{k}),
\end{align*}
which completes the proof.
\end{proof}

\begin{proposition}\label{contracting2}
Let $X=(T, \rt_T, V_\bullet)$ be a 2-colored rooted tree and $k$ an index on $X$. Assume that there exist edges $e, e'\in E$ satisfying that $e$ and $e'$ are incident at a vertex $v \in V_\circ$ with $\deg(v)=2$. Consider the quotient tree $\overline{T}:=T/e'=(\overline{V}, \overline{E})$ obtained by contracting $e$. Let $\overline{e}\in \overline{E}$, $\overline{V}_\bullet$ be the images of $e$ in $\overline{T}$ and $V_\bullet$ respectively, and $\overline{\rt_{T}} \in \overline{V}_\bullet$ be the image of $\rt_T$ in $\overline{T}$. Set $\rt_{\overline{T}}:=\overline{\rt_T}$. Let $\overline{k} : \overline{E} \rightarrow \bZ_{\geq0}$ be the index on $\overline{X}:=(\overline{T}, \rt_{\overline{T}}, \overline{V}_\bullet)$ defined by, for $f \in E$, 
$$
\overline{k}(\overline{f}):=
\begin{cases}
k(e)+k(e') & \text{if $\overline{f}=\overline{e}$,}\\
k(f) & \text{otherwise}.
\end{cases}
$$
These situations can be also written as the following figures.
$$
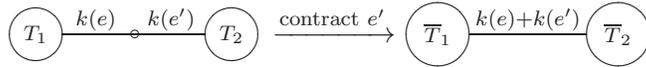
\xybox{
{(-5.53,0)*++[o][F]{\scriptstyle T_1}="1"},
{(20.53,0)*++[o][F]{\scriptstyle T_2}="2'"},
{(-2.07,0) \ar @{-{o}} (7.503,0) \ar @{-} (16.95,0)},
{(2.5,2)*++{\scriptstyle k(e)} \ar@{}(12.5,2)*++{\scriptstyle k(e')}},
}
\xrightarrow{\text{contract $e'$}}
\xybox{
{(-12,0)*++[o][F]{\scriptstyle \overline{T}_1}="1"},
{(12,0)*++[o][F]{\scriptstyle \overline{T}_2}="2'"},
{(-7.73,0) \ar @{-} (7.73,0)},
{(0,2)*++{\scriptstyle k(e)+k(e')}},
}
$$
Then we have
$$
\zeta_\cA(X, k)=\zeta_\cA(\overline{X}, \overline{k}).
$$
\end{proposition}

\begin{proof}
Since $v$ is a vertex in $V_\circ$, we have
$$
\{v \in V_\bullet \mid e \in P(\rt_T, v)\}=\{v \in V_\bullet \mid e' \in P(\rt_T, v)\}.
$$
Therefore, we obtain
\begin{align*}
\zeta_\cA(X, k)
=&\sum_{\substack{(m_v) \in \bZ^{V_\bullet}_{\geq1} \text{ s.t.} \\ \sum_{v \in V_\bullet}m_v=p}}\prod_{e\in E} L_e(\rt_T, (m_v))^{-k(e)}\\
=&\sum_{\substack{(m_v) \in \bZ^{V_\bullet}_{\geq1} \text{ s.t.} \\ \sum_{v \in V_\bullet}m_v=p}}L_e(\rt_T, (m_v))^{-(k(e)+k(e'))}\prod_{e''\in E \setminus\{e, e'\}} L_{e''}(\rt_T, (m_v))^{-k(e'')}\\
=&\sum_{\substack{(m_{\overline{v}}) \in \bZ^{\overline{V}_\bullet}_{\geq1} \text{ s.t.}\\ \sum_{\overline{v} \in \overline{V}_\bullet}m_{\overline{v}}=p}}L_{\overline{e}}(\rt_{\overline{T}}, (m_{\overline{v}}))^{-\overline{k}(\overline{e})}\prod_{\overline{f}\in \overline{E}\setminus\{\overline{e}\}} L_{\overline{f}}(\rt_{\overline{T}}, (m_{\overline{v}}))^{-\overline{k}(\overline{f})}
=\zeta_\cA(\overline{X}, \overline{k}),
\end{align*}
which completes the proof.
\end{proof}


The following proposition, which is a generalization of \cite[Lemma 3.1]{K}, is a key in obtaining non-trivial relations among the usual FMZVs.

\begin{proposition}\label{exchange_root}
For a tree $T=(V,E)$, vertices $v', v'' \in V$ and a subset $V_\bullet$ of $V$, let $X'$ (resp. $X''$) be the 2-colored rooted tree consisting of $T$, $\rt_T=v'$ (resp. $\rt_T=v''$) and $V_\bullet$. Then we have
$$
\zeta_\cA(X', k)=(-1)^{k(P(v', v''))}\zeta_\cA(X'', k)
$$
for an index $k$ on $X$. Here, $k(P(v', v'')):=\sum_{e \in P(v', v'')}k(e)$.
\end{proposition}

\begin{proof}
Consider the path $P(v', v'')$ from $v'$ to $v''$.
If $e \in P(v', v'')$, $V_\bullet$ is divided into two subsets as follows:
$$
V_\bullet=\{v \in V_\bullet \mid e \in P(v', v)\} \sqcup \{v \in V_\bullet \mid e \in P(v'', v)\}.
$$
Therefore, we have $L_e(v', (m_v))=p - L_e(v'', (m_v))$ for $e \in P(v', v'')$.
On the other hand, if $e \not\in P(v', v'')$, we see that
$$
\{v \in V_\bullet \mid e \in P(v', v)\}=\{v \in V_\bullet \mid e \in P(v'', v)\}.
$$
Thus, we have $L_e(v', (m_v))=L_e(v'', (m_v))$ for $e \not\in P(v', v'')$.
Therefore, we obtain
\begin{align*}
  &\zeta_\cA(X', k)\\
=&\sum_{\substack{(m_v) \in \bZ^{V_\bullet}_{\geq1} \text{ s.t.} \\ \sum_{v \in V_\bullet}m_v=p}}
\prod_{e\in P(v', v'')}\frac{1}{(p-L_e(v'', (m_v)))^{k(e)}}\prod_{e\not\in P(v', v'')}\frac{1}{L_e(v'', (m_v))^{k(e)}}\\
=&(-1)^{\sum_{e \in P(v', v'')}k(e)}\sum_{\substack{(m_v) \in \bZ^{V_\bullet}_{\geq1} \text{ s.t.} \\ \sum_{v \in V_\bullet}m_v=p}}
\prod_{\substack{e\in P(v', v'')}}\frac{1}{L_e(v'', (m_v))^{k(e)}}\prod_{\substack{e\not\in P(v', v'')}}\frac{1}{L_e(v'', (m_v))^{k(e)}}\\
=&(-1)^{k(P(v', v''))}\zeta_\cA(X'', k),
\end{align*}
which completes the proof.
\end{proof}

\begin{example}
Consider the 2-colored rooted tree $X'$ and the index $k$ in Example \ref{ex_of_2crt} (ii) with $\rt_T=v_{r+1}$. Then $\zeta_\cA(X', k)$ coincides with the FMZV $\zeta^{MT}_\cA(k_1, \ldots, k_r; k_{r+1})$ of MT-type. Let $X''$ be the 2-colored rooted tree whose root is $v_1$. Then, by Proposition \ref{exchange_root}, we have
\begin{align*}
\zeta^{MT}_\cA(k_1, \ldots, k_r; k_{r+1})&=\zeta_\cA(X', k)
                                                              =(-1)^{k(P(v_{r+1}, v_1))}\zeta_\cA(X'', k)\\
                                                              &=(-1)^{k_1+k_{r+1}}\zeta^{MT}_\cA(k_{r+1}, k_2, \ldots, k_r; k_1),
\end{align*}
which is Kamano's result \cite[Lemma 3.1]{K}.
\end{example}

For the proof of our main theorem, we need the following definitions that the pair consisting of a 2-colored rooted tree and an index on it is harvestable and that an index on a 2-colored rooted tree is essentially positive.

\begin{definition}\label{def_of_hv_form}

Let $X=(T, \rt_T, V_\bullet)$ be a 2-colored rooted tree and $k$ an index on $X$. The pair $(X, k)$ is harvestable if the following conditions on $(X,k)$ holds.

\begin{description}

\item[(H1)] The root $\rt_T$ is a terminal of $T$. In particular, $\rt_T$ is in $V_\bullet$.

\item[(H2)] $\deg(v) \leq 2$ for all $v$ in $V_\bullet$ and $\deg(v) \geq3$ for all $v$ in $V_\circ$.

\item[(H3)] If an edge $e$ connects a branched point $v$ in $V_\circ$ and a child of $v$ in $V_\bullet$, $k(e)$ is positive. Here, a branched point is a vertex whose degree is larger than or equal to 3.

\end{description}

\end{definition}

\begin{definition}\label{adm}
For a 2-colored rooted tree $X=(T, \rt_T, V_\bullet)$, an index $k$ on $X$ is essentially positive if
$k(P(v, v'))$ is positive for any $v, v' \in V_\bullet$.
\end{definition}

\begin{remark}
If we cut the edges of 2-colored rooted trees which connect the branched point nearest to the root, the 2-colored rooted tree becomes decomposed into many parts. We take the upper part, and the parts under the branched point again become 2-colored rooted trees satisfying the conditions \textbf{(H1)}, \textbf{(H2)} and \textbf{(H3)} after adding new roots to each part.  We call this operation of taking the upper part as “harvest”, and this is the reason why we call 2-colored rooted trees satisfying \textbf{(H1)}, \textbf{(H2)} and \textbf{(H3)} as being harvestable. 
\end{remark}




By using Proposition \ref{contracting1}, Proposition \ref{contracting2} and Proposition \ref{exchange_root}, we see that for a pair consisting of a 2-colored rooted tree and an essentially positive index on it, there exists a harvestable pair such that FMZVs associated with them coincides up to sign.


\begin{proposition}\label{hv_form}
Let $X=(T, \rt_T, V_\bullet)$ be a 2-colored rooted tree and $k$ an essentially positive index on $X$. Then, there exists a harvestable pair $(X_\h, k_\h)$ of the 2-colored rooted tree $X_\h=(T_\h=(V(X_\h), E(X_\h)), \rt_{T_\h}, V_\bullet(X_\h))$ and the index $k_\h$ on $X_\h$ satisfying 
\begin{equation}\label{hv}
\zeta_\cA(X, k)=(-1)^{k_{\h}(P(\overline{\rt_T}, \rt_{T_\h}))}\zeta_\cA(X_\h, k_\h).
\end{equation}
Here $\overline{\rt_T} \in V(X_\h)$ is the image of $\rt_T$ in $X_\h$.
\end{proposition}

\begin{proof}
By using Proposition \ref{contracting1} and \ref{contracting2} to contract edges that one of end points is in $V_\circ$ and $k(e)=0$, and edges connecting $v' \in V_\circ$ with $\deg(v')=2$, and again using Proposition \ref{contracting1} to insert edges $e'$ with $k(e')=0$ and vertices in $V_\circ$ into vertices in $V_\bullet$ whose degree is greater than or equal to 3, we obtain a pair $(X', k')$ of a 2-colored rooted tree $X'$ and an index $k'$ on $X'$ satisfying the conditions \textbf{(H2)}, \textbf{(H3)} and that FMZVs associated with them coincides. Further, by using Proposition \ref{exchange_root} to move the root $\rt_T$ to a terminal, we obtain a desired harvestable pair $(X_\h, k_\h)$ satisfying $\zeta_\cA(X, k)=(-1)^{k_\h(P(\overline{\rt_T}, \rt_{T_\h}))}\zeta_\cA(X_\h, k_\h)$, which completes the proof.
\end{proof}

\begin{definition}
For a 2-colored rooted tree $X$ and an essentially positive index $k$ on $X$, we define a harvestable form of the pair $(X,k)$ as the harvestable pair $(X_\h, k_\h)$ satisfying the Proposition \ref{hv_form}.
\end{definition}


\begin{remark}
For a 2-colored rooted tree $X$ and an essentially positive index $k$ on $X$, a harvestable pair $(X_\h, k_\h)$ of $(X,k)$ is not unique. For example, consider the following 2-colored rooted tree $X=(T, \rt_T, V_\bullet)$ and an essentially positive index $k$ on $X$.
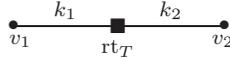
$$
\xybox{
{(-14,0) \ar @{{*}-} (0,0) \ar @{-{*}} (14,0)},
{(0,0)*++{\scriptstyle \blacksquare} \ar @{} (-13,-2)*++{\scriptstyle v_1} \ar@{}(0,-3)*++{\scriptstyle \rt_T} \ar@{}(14,-2)*++{\scriptstyle v_2}\ar@{}(-7,2)*++{\scriptstyle k_1} \ar@{}(7,2)*++{\scriptstyle k_2} },
}
$$
Then, we can take the following two 2-colored rooted trees and indices as harvestable forms of $(X,k)$.
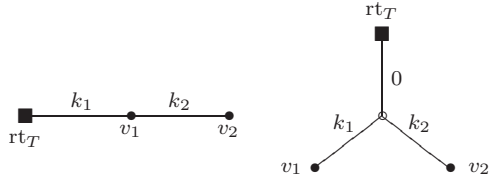
$$
\xybox{
{(-14,0) \ar @{-{*}} (0,0) \ar @{-{*}} (13,0)},
{(-14,0)*++{\scriptstyle \blacksquare} \ar @{}(-14,-3)*++{\scriptstyle \rt_T} \ar@{}(0,-2)*++{\scriptstyle v_1} \ar@{}(13,-2)*++{\scriptstyle v_2}\ar@{}(-6.5,2)*++{\scriptstyle k_1} \ar@{}(6.5,2)*++{\scriptstyle k_2}},
}
\xybox{
{(-9,-7) \ar @{{*}-{o}} (0,0)},
{(9,-7)\ar@{{*}-}(0,0)},
{(0,10)\ar@{-}(0,0)},
{(0,11)*++{\scriptstyle \blacksquare} \ar @{} (-12,-7)*++{\scriptstyle v_1} \ar@{}(0,14)*++{\scriptstyle \rt_T} \ar@{}(13,-7)*++{\scriptstyle v_2}\ar@{}(-5,-1)*++{\scriptstyle k_1} \ar@{}(5,-1)*++{\scriptstyle k_2} \ar@{}(2,5)*++{\scriptstyle 0}},
}
$$
\end{remark}

\section{Proof of our main theorem}

In this section, we will prove our main theorem (Theorem \ref{mt}). If the pair $(X, k)$ consisting of a 2-colored rooted tree $X$ and an essentially positive index $k$ on $X$ is harvestable, our main theorem will be proved by induction on the sum of the index at the edges in paths from the branched point nearest to the root to all leaves (Proposition \ref{main theorem}). The general case will be deduced to the harvestable case by using Proposition \ref{hv_form}.

To explain that $\zeta_\cA(X, k)$ associated with a harvestable pair $(X, k)$ can be written explicitly as a $\bZ$-linear combination of the usual FMZVs, we prepare some terminology on Hoffman's algebra \cite{H1}. Let $\frH$ be the noncommutative polynomial ring $\bQ \langle x, y \rangle$ in the variables $x$ and $y$ over $\bQ$, 
and we set $\frH^1:=\bQ + y\frH$. Note that $\frH^1$ is generated by $z_k:=yx^{k-1} \; (k=1, 2, \ldots)$ as a $\bQ$-algebra.
We define the map $Z_\cA : \frH^1 \rightarrow \cA$ by sending $z_{k_1}\cdots z_{k_r}$ to $\zeta_\cA(k_1, \ldots, k_r)$ and extend it $\bQ$-linearly.

\begin{definition} \label{shuffle product}
We define the shuffle product $\sh : \frH \times \frH \rightarrow \frH$ on $\frH$ by the following rule and $\bQ$-bilinearity.

\begin{enumerate}
\renewcommand{\labelenumi}{(\roman{enumi})}

\item $w \sh 1=1 \sh w =w$ for all $w \in \frH$.

\item $(w_1u_1)\sh(w_2u_2)=(w_1\sh w_2u_2)u_1+(w_1u_1\sh w_2)u_2$ for all $w_1, w_2 \in \frH$ and $u_1, u_2 \in \{x, y\}$.

\end{enumerate}
\end{definition}

For instance, we have
$$
z_2 \sh z_2=yx\sh yx=2yxyx+4y^2x^2=2z_2z_2+4z_1z_3.
$$

The next proposition is the harvestable case of our main theorem.

\begin{proposition}\label{main theorem}
Let $X=(T, \rt_T, V_\bullet)$ be a 2-colored rooted tree and $k$ be an essentially positive index on $X$. Assume that the pair $(X,k)$ is harvestable.
Then $\zeta(X,k)$ coincides with the image $Z_\cA(w)$ of $w \in \frH^1$ constructed by the following inductive method.
Since $(X, k)$ is a harvestable pair, $(X, k)$ has the following shape.
$$
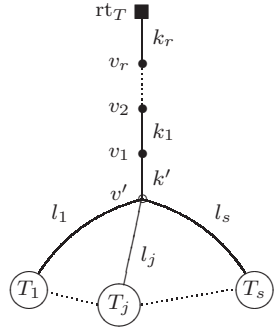
\xybox{
{(-15,-10)*+[o][F]{\scriptstyle T_1}="1"},
{(-3,-12)*+[o][F]{\scriptstyle T_j}="j"},
{(15,-10)*+[o][F]{\scriptstyle T_s}="s"},
{"1" \ar^{l_1} @/^/@{-} (0,2)},
{"j" \ar_{l_j} @{-} (0,2)}, 
{"s" \ar_{l_s} @/_/ @{-} (0,2)}, 
{(0,2) \ar_{k'} @{o-{*}} (0,8)}, 
{(0,8) \ar_{k_1} @{-{*}} (0,14)}, 
{(0,14) \ar @{.} (0,20)}, 
{(0,20) \ar_{k_r} @{{*}-} (0,27)}, 
{(0,-12) \ar @{.} (12,-10)}, 
{(-12,-10.5) \ar @{.} (-6,-12)}, 
{(0,27)*++{\scriptstyle \blacksquare} \ar @{}(-3,3)*++{\scriptstyle v'} \ar @{}(-3,8)*++{\scriptstyle v_1} \ar @{}(-3,14)*++{\scriptstyle v_2} \ar @{}(-3,20)*++{\scriptstyle v_r} \ar @{}(-4,27)*++{\scriptstyle \rt_T}},
}
$$
Here, $r$ and $s$ are positive integers and $k_i:=k(e_i) \; (1\leq i \leq r), l_j:=k(f_j) \; (1\leq j\leq s)$ and $k':=k(e')$. $T_1, \ldots, T_s$ are subtree of $T$. Then we have
$$
w=\Bigl(\foo_{j=1}^sw_j\Bigr)x^{k'}z_{k_1}\cdots z_{k_r}.
$$
Here, $w_j \in \frH^1 \; (1\leq j \leq s)$ are elements corresponding to the following harvestable pair $(X_j, k^{(j)})$.
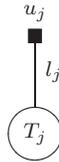
$$
\xybox{
{(0,-10)*++[o][F]{\scriptstyle T_j}="j"},
{(0,-6.5) \ar_{l_j} @{-} (0,3)},
{(0,3)*++{\scriptstyle \blacksquare} \ar @{} (0,6)*++{\scriptstyle u_j}},
}
$$
\end{proposition}

For the proof of the above proposition, the following partial fraction decomposition is the key tool. 
For a positive integer $s$ and indeterminates $X_1, \ldots, X_s$, we have
\begin{equation}\label{eq:PFD}
\frac{1}{X_1\cdots X_s}=\frac{1}{X_1+\cdots+X_s}\sum_{i=1}^s\underbrace{\frac{1}{X_1\cdots X_s}}_{\text{remove $i$-th}}.
\end{equation}

\begin{proof}[Proof of Proposition \ref{main theorem}]
Let $v' \in V$ be the branched point nearest to the root $\rt_T$ in all the branched points. As $(X, k)$ is harvestable, $v'$ is an element in $V_\circ$. Set $S=S(X,k):=\sum_ek(e)$, where $e$ runs through edges in paths from $v'$ to all leaves.
If there exists no branched points, we set $S=0$. 
We prove the statement by the induction on $S \geq0$.
If $S=0$, as $(X, k)$ is harvestable, we see that $\zeta_\cA(X, k)$ coincides with that associated with the following 2-colored rooted tree and the index on it with the root $v_{r+1}$:
$$
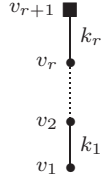
\xybox
{ 
{(0,-8) \ar_{k_1} @{{*}-{*}} (0,-2)},
{(0,-2) \ar @{.{*}} (0,6)},
{(0,6) \ar_{k_r} @{-} (0,13)},
{(0,13)*++{\scriptstyle \blacksquare} \ar @{}(-3,-8)*++{\scriptstyle v_1} \ar @{} (-3,-2)*++{\scriptstyle v_2} \ar @{} (-3,6)*++{\scriptstyle v_r} \ar @{} (-5,13)*++{\scriptstyle v_{r+1}}},
}
$$
Therefore, we have $\zeta_\cA(X, k)=\zeta_\cA(k_1, \ldots, k_r)=Z_\cA(z_{k_1}\cdots z_{k_r})$,
which completes the proof of the case $S=0$.
Next, assume that $S>0$ and the statement holds for all the non-negative integers less than $S$.
The assumption $S>0$ means that there exists at least one branched point.
Then the given 2-colored rooted tree $X=(T, \rt_T, V_\bullet)$ and the index $k$ on $X$ can be written as follows:
$$
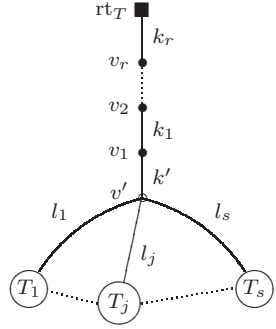
\xybox{
{(-15,-10)*+[o][F]{\scriptstyle T_1}="1"},
{(-3,-12)*+[o][F]{\scriptstyle T_j}="j"},
{(15,-10)*+[o][F]{\scriptstyle T_s}="s"},
{"1" \ar^{l_1} @/^/@{-} (0,2)},
{"j" \ar_{l_j} @{-} (0,2)}, 
{"s" \ar_{l_s} @/_/ @{-} (0,2)}, 
{(0,2) \ar_{k'} @{o-{*}} (0,8)}, 
{(0,8) \ar_{k_1} @{-{*}} (0,14)}, 
{(0,14) \ar @{.} (0,20)}, 
{(0,20) \ar_{k_r} @{{*}-} (0,27)}, 
{(0,-12) \ar @{.} (12,-10)}, 
{(-12,-10.5) \ar @{.} (-6,-12)}, 
{(0,27)*++{\scriptstyle \blacksquare} \ar @{} (-3,3)*++{\scriptstyle v'} \ar @{}(-3,8)*++{\scriptstyle v_1} \ar @{}(-3,14)*++{\scriptstyle v_2} \ar @{}(-3,20)*++{\scriptstyle v_r} \ar @{}(-4,27)*++{\scriptstyle \rt_T}},
}
$$
Then we have
\begin{multline*}
   \zeta_\cA(X, k)
=\sum_{\substack{(m_v)\in \bZ^{V_\bullet}_{\geq1} \text{ s.t.} \\ \sum_{v \in V_\bullet}m_v=p}}
\prod_{i=1}^r\frac{1}{L_{e_i}(\rt_T, (m_v))^{k_i}}
\times \frac{1}{L_{e'}(\rt_T, (m_v))^{k'}}\\
\times \prod_{j=1}^s\prod_{e \in E(T_j)}\frac{1}{L_e(\rt_T, (m_v))^{k(e)}}
\times \frac{1}{M^{l_1}_1\cdots M^{l_s}_s}.
\end{multline*}
Here, for $1 \leq j \leq s$, we set
$$
M_j:=\sum_{\substack{v \in V_\bullet \text{ s.t. } \\ f_j \in P(\rt_T, v)}}m_v = L_{f_j}(\rt_T, (m_v)).
$$
By \eqref{eq:PFD}, we have
\begin{align*}
   &\frac{1}{M^{l_1}_1\cdots M^{l_s}_s}
=\frac{1}{M_1+\cdots+M_s}
\left(\frac{1}{M^{l_1-1}_1M^{l_2}_2\cdots M^{l_s}_s}
+\cdots+\frac{1}{M^{l_1}_1\cdots M^{l_{s-1}}_{s-1}M^{l_s-1}_s}\right) .
\end{align*}
Therefore, since $L_{e'}(\rt_T, (m_v))=M_1+\cdots+M_s$, we obtain
$$
\zeta_\cA(X, k)=\sum_{j=1}^s\zeta_\cA(X, \alpha_j).
$$
Here, $\alpha_j$ is an index on $X$ defined by
$$
\alpha_j(e)=
\begin{cases}
l_j-1 & e=f_j, \\
k'+1 & e=e', \\
k(e) & \text{otherwise,}
\end{cases}
$$
and 
\begin{multline*}
   \zeta_\cA(X, \alpha_j)
=\sum_{\substack{(m_v)\in \bZ^{V_\bullet}_{\geq1} \text{ s.t.} \\ \sum_{v \in V_\bullet}m_v=p}}
\prod_{i=1}^r\frac{1}{L_{e_i}(\rt_T, (m_v))^{k_i}}
\times \frac{1}{L_{e'}(\rt_T, (m_v))^{k'+1}}\\
\times \prod_{j=1}^s\prod_{e \in E(T_j)}\frac{1}{L_e(\rt_T, (m_v))^{k(e)}}
\times \frac{1}{M^{l_1}_1\cdots M^{l_j-1}_j \cdots M^{l_s}_s}.
\end{multline*}

To use the induction hypotheses, we need to consider two cases whether $(X, \alpha_j)$ is harvestable or not.
\begin{enumerate}
\item First, consider the case that the pair $(X,\alpha_j)$ is harvestable. This is the case when $l_j>1$ or the child of $v'$ incident to $f_j$ is in $V_\circ$. In this case, the pair $(X,\alpha_j)$ has the following shape.
$$
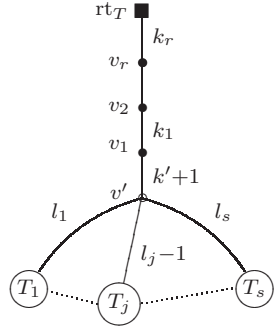
\xybox{
{(-15,-16)*+[o][F]{\scriptstyle T_1}="1"},
{(-3,-18)*+[o][F]{\scriptstyle T_j}="j"},
{(15,-16)*+[o][F]{\scriptstyle T_s}="s"},
{"1" \ar^{l_1} @/^/@{-} (0,-4)},
{"j" \ar_{l_j-1} @{-} (0,-4)}, 
{"s" \ar_{l_s} @/_/ @{-} (0,-4)}, 
{(0,-4) \ar_{k'+1} @{o-} (0,2)}, 
{(0,2) \ar_{k_1} @{{*}-{*}} (0,8)}, 
{(0,8) \ar @{-{*}} (0,14)}, 
{(0,8) \ar @{.} (0,14)}, 
{(0,14) \ar_{k_r} @{{*}-} (0,21)}, 
{(0,-18) \ar @{.} (12,-16)}, 
{(-12,-16.5) \ar @{.} (-6,-18)}, 
{(0,21)*++{\scriptstyle \blacksquare} \ar @{} (-3,-3)*++{\scriptstyle v'} \ar @{}(-3,3)*++{\scriptstyle v_1} \ar @{}(-3,8)*++{\scriptstyle v_2} \ar @{}(-3,14)*++{\scriptstyle v_r} \ar @{}(-4,21)*++{\scriptstyle \rt_T}},
}
$$
Since $S(X, \alpha_j)=S(X, k)-1<S(X,k)$, by the induction hypotheses, we obtain
$$
\zeta_\cA(X, \alpha_j)=Z_\cA\left(\Bigl(\foo_{\substack{a=1\\a \neq j}}^sw_a\sh w'_j\Bigr)x^{k'+1}z_{k_1}\cdots z_{k_r}\right).
$$
%
Here $w'_j$ is the element of $\frH^1$ corresponding to the following harvestable pair $(X_j, \beta_j)$.
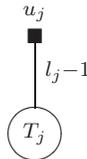
$$
\xybox{
{(0,-10)*++[o][F]{\scriptstyle T_j}="j"},
{(0,-6.5) \ar_{l_j-1} @{-} (0,3)},
{(0,3)*++{\scriptstyle \blacksquare} \ar @{} (0,6)*++{\scriptstyle u_j}},
}
$$
Note that $w_j=w'_jx$ in this case.

\item
Next, consider the case that the pair $(X, \alpha_j)$ is not harvestable. This is the case when $l_j=1$ and the child of $v'$ incident to $f_j$ is in $V_\bullet$ because $\alpha_j(f_j)=l_j-1=0$.
By using Proposition \ref{contracting1} to contract $f_j$ and insert edges $e'$ with $\alpha_j(e')=0$, we obtain a harvestable form $(X_\h, \alpha_{j,\h})$ of $(X, \alpha_j)$ as follows.
$$
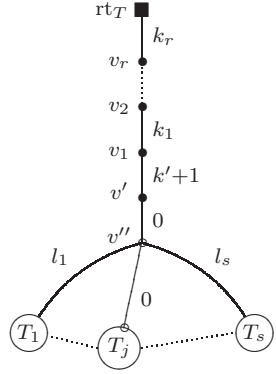
\xybox{
{(-15,-16)*+[o][F]{\scriptstyle T_1}="1"},
{(-3,-18)*+[o][F]{\scriptstyle T_j}="j"},
{(15,-16)*+[o][F]{\scriptstyle T_s}="s"},
{"1" \ar^{l_1} @/^/@{-} (0,-4)},
{"j" \ar_{0} @{o-} (0,-4)}, 
{"s" \ar_{l_s} @/_/ @{-} (0,-4)}, 
{(0,-4) \ar_{0} @{o-} (0,2)}, 
{(0,2) \ar_{k'+1} @{{*}-{*}} (0,8)}, 
{(0,8) \ar_{k_1} @{-{*}} (0,14)}, 
{(0,14) \ar @{.} (0,20)}, 
{(0,20) \ar_{k_r} @{{*}-} (0,27)}, 
{(0,-18) \ar @{.} (12,-16)}, 
{(-12,-16.5) \ar @{.} (-6,-18)}, 
{(0,27)*++{\scriptstyle \blacksquare} \ar @{} (-3,-3)*++{\scriptstyle v''} \ar @{}(-3,3)*++{\scriptstyle v'} \ar @{}(-3,8)*++{\scriptstyle v_1} \ar @{}(-3,14)*++{\scriptstyle v_2} \ar @{}(-3,20)*++{\scriptstyle v_r} \ar @{}(-4,27)*++{\scriptstyle \rt_T}},
}
$$
Since $S(X, \alpha_j)=S(X_\h, \alpha_{j,\h})=S(X,k)-1<S(X,k)$, by the induction hypotheses, we obtain 
\begin{align*}
\zeta_\cA(X, \alpha_j)=\zeta_\cA(X_\h, \alpha_{j,\h})&=Z_\cA\left(\Bigl(\foo_{\substack{a=1\\a \neq j}}^sw_a\sh w'_j\Bigr)x^0z_{k'+1}z_{k_1}\cdots z_{k_r}\right)\\
                               &=Z_\cA\left(\Bigl(\foo_{\substack{a=1\\a \neq j}}^sw_a\sh w'_j\Bigr)yx^{k'}z_{k_1}\cdots z_{k_r}\right).
\end{align*}
Here $w'_j$ is the element of $\frH^1$ corresponding to the following harvestable pair $(X'_j, \beta_j)$.
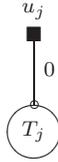
$$
\xybox{
{(0,-10)*++[o][F]{\scriptstyle T_j}="j"},
{(0,-6.5) \ar_{0} @{{o}-} (0,3)},
{(0,3)*++{\scriptstyle \blacksquare} \ar @{} (0,6)*++{\scriptstyle u_j}},
}
$$
Note also that $w_j=w'_jy$ in this case.
\end{enumerate}
Therefore, by the definition of the shuffle product, we obtain
$$
\zeta_\cA(X, k)=\sum_{j=1}^s\zeta_\cA(X, \alpha_j)
=Z_\cA\left(\Bigl(\foo_{j=1}^sw_j\Bigr)x^{k'}z_{k_1}\cdots z_{k_r}\right).
$$
Therefore, $w:=(\foo_{j=1}^sw_j)x^{k'}z_{k_1}\cdots z_{k_r}$ is the desired element in $\frH^1$.
\end{proof}

\begin{proof}[Proof of Theorem \ref{mt}]
By Proposition \ref{hv_form}, for a given pair $(X,k)$ consisting of a 2-colored rooted tree $X$ and an essentially positive index $k$ on $X$, there exists a harvestable pair $(X_\h, k_\h)$ satisfying \eqref{hv}. Since the pair $(X_\h, k_\h)$ is harvestable, by Proposition \ref{main theorem}, the right hand side of \eqref{hv} can be written explicitly as a $\bZ$-linear combination of the usual FMZVs, so can $\zeta_\cA(X, k)$. This completes the proof of our main theorem.
\end{proof}

\begin{example}\label{eg}
\begin{enumerate}

\item For $1\leq i \leq r$, consider the following 2-colored rooted tree $X$ and the essentially positive index $k$ on $X$.
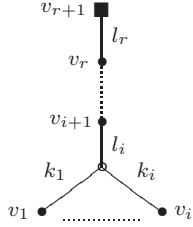
          $$
          \xybox{
          {(-8,-4) \ar^{k_1} @{{*}-o} (0,2)}, 
          {(8,-4) \ar_{k_i} @{{*}-o} (0,2)}, 
          {(0,2) \ar_{l_i} @{-{*}} (0,8)}, 
          {(0,8) \ar @{.{*}} (0,16)},
          {(0,16) \ar_{l_r} @{-} (0,23)},
          {(-5,-5) \ar @{.} (5,-5)}, 
          {(0,23)*++{\scriptstyle \blacksquare} \ar @{} (-11,-4)*++{\scriptstyle v_1} \ar @{} (11,-4)*++{\scriptstyle v_i} \ar @{} (-4,8)*++{\scriptstyle v_{i+1}} \ar@{} (-3,16)*++{\scriptstyle v_r} \ar@{} (-5,23)*++{\scriptstyle v_{r+1}}}, 
          }
          $$
          Set $\rt_T=v_{r+1}$. Since $k$ is essentially positive, $k_j\;(1\leq j \leq i)$ and $l_j\;(i+1\leq j \leq r)$ are positive.
          Then we have
          \begin{align*}
            \zeta_\cA(X, k)
          =&\sum_{\substack{m_1, \ldots, m_r \geq 1\\ m_1+\cdots+m_r \leq p-1}}
          \frac{1}{m^{k_1}_1\cdots m^{k_i}_i(m_1+\cdots+m_i)^{l_i}\cdots(m_1+\cdots+m_r)^{l_r}}\\
          =&Z_\cA((z_{k_1} \sh \cdots \sh z_{k_i})x^{l_i}z_{l_{i+1}}\cdots z_{l_r}),
          \end{align*}
          which is Kamano's result \cite[Theorem 2.1]{K}.
          
\item Consider the following 2-colored tree $X$ and the essentially positive index $k$ on $X$.
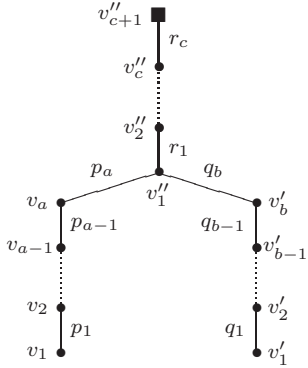
          \[\begin{xy}
          {(-13,-20) \ar_{p_1} @{{*}-} (-13,-14)}, 
          {(-13,-14) \ar @{{*}.{*}} (-13,-6)},
          {(-13,-6) \ar_{p_{a-1}} @{{*}-{*}} (-13,0)}, 
          {(-16, -20)*++{\scriptstyle v_1} \ar @{} (-16,-14)*++{\scriptstyle v_2} \ar @{} (-17,-6)*++{\scriptstyle v_{a-1}} \ar @{} (-16,0)*++{\scriptstyle v_a}}, 
          {(-13,0) \ar^{p_a} @{-} (0,4)}, 
          {(13,-20) \ar^{q_1} @{{*}-} (13,-14)}, 
          {(13,-14) \ar @{{*}.{*}} (13,-6)},
          {(13,-6) \ar^{q_{b-1}}@{{*}-{*}} (13,0)},
          {(16,-20)*++{\scriptstyle v'_1} \ar @{} (16,-14)*++{\scriptstyle v'_2} \ar @{} (17,-6)*++{\scriptstyle v'_{b-1}} \ar @{} (16,0)*++{\scriptstyle v'_b}}, 
          {(13,0) \ar_{q_b} @{-} (0,4)}, 
          {(0,4)*+!U{\scriptstyle v''_1} \ar_{r_1} @{{*}-{*}} (0,10)}, 
          {(0,10) \ar @{{*}.{*}}(0,18)},
          {(0,18) \ar_{r_c} @{{*}-} (0,25)},
          {(0,25)*++{\scriptstyle \blacksquare} \ar @{}(-3,10)*++{\scriptstyle v''_2} \ar @{} (-3,18)*++{\scriptstyle v''_c} \ar @{} (-5,25)*++{\scriptstyle v''_{c+1}}}, 
          \end{xy}\]
          Here, $a, b$ and $c$ are non-negative integers. 
          Set $\rt_T=v''_{c+1}$. Since $k$ is essentially positive, $p_x \; (1\leq x \leq a)$, $q_y \; (1\leq y \leq b)$ and $r_z \; (1\leq z \leq c)$ are all positive. Then we have 
          \begin{align*}
          \zeta_\cA(X, k)=&\sum_{\substack{0< l_1<\cdots<l_a \\ 0< m_1<\cdots<m_b \\ l_a+m_b<n_1<\cdots<n_c<p}}
\frac{1}{l^{p_1}_1\cdots l^{p_a}_am^{q_1}_1\cdots m^{q_b}_bn^{r_1}_1\cdots n^{r_c}_c}\\
          =& Z_\cA((z_{p_1}\cdots z_{p_a}{\sh}z_{q_1}\cdots z_{q_b})z_{r_1}\cdots z_{r_c}),
          \end{align*} which is a finite analogue of a result of Komori, Matsumoto and Tsumura \cite[Theorem 1]{KMT}.

\end{enumerate}
\end{example}

\begin{remark}
Since the only tool used to prove the main theorem (Theorem \ref{mt}) is the partial fraction decomposition
\eqref{eq:PFD}, the analogous statement of Theorem \ref{mt} for the MZVs holds. 
For example, we obtain a result 
\begin{equation}\label{eq:A-type}
\sum_{\substack{0_<l_1<\cdots<l_a \\ 0<m_1<\cdots<m_b \\ l_a+m_b<n_1<\cdots<n_c}}
\frac{1}{l^{p_1}_1\cdots l^{p_a}_am^{q_1}_1\cdots m^{q_b}_bn^{r_1}_1\cdots n^{r_c}_c}
=Z((z_{p_1} \cdots z_{p_a} \sh z_{q_1} \cdots z_{q_b})z_{r_1}\cdots z_{r_c})
\end{equation}
of Komori-Matsumoto-Tsumura \cite[Theorem 1]{KMT} by using our method. Here, 
\begin{equation*}
Z : \frH^0:=\bQ + y\frH x \rightarrow \bR \; ; \;\; z_{k_1}\cdots z_{k_r} \mapsto \zeta(k_1, \ldots, k_r)
\end{equation*}
is a $\bQ$-linear map. The left hand side of \eqref{eq:A-type} can be regarded as a special value of 
the multiple zeta function $\zeta(\bs; A_r)$ of the root system of type $A_r$
$$
\zeta(\bs; A_r)
:=\sum_{m_1, \ldots, m_r \geq 1}\prod_{1\leq i<j \leq r+1}(m_i+\cdots+m_{j-1})^{-s_{ij}},
$$
which was first defined by Matsumoto and Tsumura in \cite{MT}. Indeed, we have
$$
\sum_{\substack{0_<l_1<\cdots<l_a \\ 0<m_1<\cdots<m_b \\ l_a+m_b<n_1<\cdots<n_c}}
\frac{1}{l^{p_1}_1\cdots l^{p_a}_am^{q_1}_1\cdots m^{q_b}_bn^{r_1}_1\cdots n^{r_c}_c}
=Z((z_{p_1} \cdots z_{p_a} \sh z_{q_1} \cdots z_{q_b})z_{r_1}\cdots z_{r_c})
=\zeta((k_{i,j}); A_r)
$$
for 
$$
k_{ij}=
\begin{cases}
p_{j-1} & i=1,2 \leq j \leq a+1, \\
q_{j-(a+1)} & i=a+1, a+2 \leq j \leq a+b+1, \\
r_{j-(a+b+1)} & i=1, a+b+2 \leq j \leq a+b+c+1, \\
0 & \text{otherwise}.
\end{cases}
$$
\end{remark}

\section{Applications for non-trivial relations among FMZVs}

In the final section, using the Theorem \ref{mt} and the Proposition \ref{exchange_root}, we give another proof of the shuffle relation among FMZVs, which was first proved by Kaneko and Zagier in \cite{KZ}.

\begin{corollary}([KZ])\label{sh_prod_FMZV}
For positive integers $k_1, \ldots, k_r, l_1, \ldots, l_s$ and elements $w:=z_{k_1}\cdots z_{k_r}, w':=z_{l_1}\cdots z_{l_s} \in \frH^1$, we have
$$
Z_\cA(w \sh w')=(-1)^{l_1+\cdots+l_s}Z_\cA(z_{k_1}\cdots z_{k_r}z_{l_s}\cdots z_{l_1}).
$$
\end{corollary}

\begin{proof}
Consider the following two 2-colored rooted trees $X, X'$, whose root are $v$ and $v'_1$, and index $k$ on $X$ and $X'$. 
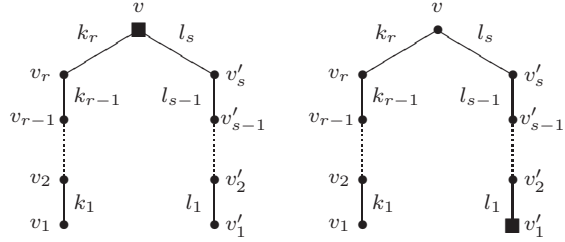
          \[\begin{xy}
          {(-10,-26) \ar_{k_1} @{{*}-} (-10,-20)}, 
          {(-10,-20) \ar @{{*}.{*}} (-10,-12)},
          {(-10,-12) \ar_{k_{r-1}} @{-{*}} (-10,-6)},
          {(-13,-26)*++{\scriptstyle v_1} \ar @{} (-13,-20)*++{\scriptstyle v_2}\ar @{} (-14,-12)*++{\scriptstyle v_{r-1}} \ar @{} (-13,-6)*++{\scriptstyle v_r}}, 
          {(-10,-6) \ar^{k_r} @{-{*}} (0,0)}, 
          {(10,-26) \ar^{l_1} @{{*}-} (10,-20)}, 
          {(10,-20) \ar @{{*}.{*}} (10,-12)},
          {(10,-12) \ar^{l_{s-1}} @{-{*}} (10,-6)},
          {(0,0)*++{\scriptstyle \blacksquare} \ar @{} (13,-26)*++{\scriptstyle v'_1} \ar @{} (13,-20)*++{\scriptstyle v'_2}\ar @{} (14,-12)*++{\scriptstyle v'_{s-1}} \ar @{} (13,-6)*++{\scriptstyle v'_s}}, 
          {(10,-6) \ar_{l_s} @{-} (0,0)},  
          {(0,1)*+!D{\scriptstyle v}}, 
          \end{xy}
           \begin{xy}
          {(-10,-26) \ar_{k_1} @{{*}-} (-10,-20)}, 
          {(-10,-20) \ar @{{*}.{*}} (-10,-12)},
          {(-10,-12) \ar_{k_{r-1}} @{-{*}} (-10,-6)},
          {(-13,-26)*++{\scriptstyle v_1} \ar @{} (-13,-20)*++{\scriptstyle v_2}\ar @{} (-14,-12)*++{\scriptstyle v_{r-1}} \ar @{} (-13,-6)*++{\scriptstyle v_r}}, 
          {(-10,-6) \ar^{k_r} @{-{*}} (0,0)}, 
          {(10,-26) \ar^{l_1} @{{*}-} (10,-20)}, 
          {(10,-20) \ar @{{*}.{*}} (10,-12)},
          {(10,-12) \ar^{l_{s-1}} @{-{*}} (10,-6)},
          {(13,-26)*++{\scriptstyle v'_1} \ar @{} (10,-26)*++{\scriptstyle \blacksquare} \ar @{}(13,-20)*++{\scriptstyle v'_2}\ar @{} (14,-12)*++{\scriptstyle v'_{s-1}} \ar @{} (13,-6)*++{\scriptstyle v'_s}}, 
          {(10,-6) \ar_{l_s} @{-} (0,0)},  
          {(0,1)*+!D{\scriptstyle v}}, 
          \end{xy}\]
Then, by Proposition \ref{exchange_root}, we have
\begin{equation}\label{eq:change roots}
\zeta_\cA(X, k)=(-1)^{k(P(v, v'_1))}\zeta_\cA(X', k).
\end{equation}
By Theorem \ref{mt}, the left hand side of \eqref{eq:change roots} coincides with
$$
Z_\cA(w \sh w').
$$
On the other hand, by Proposition \ref{main theorem},
the right hand side of \eqref{eq:change roots} coincides with
$$
(-1)^{l_1+\cdots+l_s}Z_\cA(z_{k_1}\cdots z_{k_r}z_{l_s}\cdots z_{l_1}).
$$
Therefore, we obtain the shuffle relation among FMZVs.
\end{proof}

\begin{remark}
The case $r=s=1, k_1=1$ and $l_1=k-1$ for $k>1$, the Proposition \ref{sh_prod_FMZV} says that
\begin{equation}\label{eq:shprod}
Z_\cA(z_1\sh z_{k-1})=-Z_\cA(z_{k-1}z_1).
\end{equation}
The right hand side of \eqref{eq:shprod} is $-\zeta_\cA(k-1, 1)$, which is equal to $B_{\bp-k}$ by Hoffman's result \cite[Theorem 6.1]{H2}. Here $B_{\bp-k}:=(B_{p-k})_p \in \cA$ and $B_n$ is the $n$-th Bernoulli number. On the other hand, since
$$
z_1\sh z_{k-1}=z_1z_{k-1}+\sum_{\substack{k_1, k_2 \geq 1\\ k_1+k_2=k}}z_{k_1}z_{k_2},
$$
the left hand side of $\eqref{eq:shprod}$ is equal to
$$
\zeta_\cA(1,k-1)+\sum_{\substack{k_1, k_2 \geq 1\\ k_1+k_2=k}}\zeta_\cA(k_1, k_2).
$$
Therefore, by \cite[Theorem 4.4]{H2}, we have
$$
\sum_{\substack{k_1, k_2 \geq 1\\ k_1+k_2=k}}\zeta_\cA(k_1, k_2)=-(\zeta_\cA(k-1,1)+\zeta_\cA(1,k-1))=0.
$$ 
This equality is equivalent to the sum formula for double FMZVs \cite[Theorem 1.4]{SW1}.
Indeed, by \cite[Theorem 1.4]{SW1} and \cite[Theorem 6.1]{H2}, we have
$$
\sum_{\substack{k_1, k_2 \geq 1, k_i \geq2 \\ k_1+k_2=k}}\zeta_\cA(k_1, k_2)=(-1)^{k+i}B_{\bp-k}=
\begin{cases}
-\zeta_\cA(1,k-1) & \text{if $i=1$,} \\
-\zeta_\cA(k-1,1) & \text{if $i=2$}.
\end{cases}
$$
\end{remark}

\section*{Acknowledgments}
The starting point of this study was when the author read the manuscript of \cite{K} at the conference ``Zeta Functions of Several Variables and Applications", held at Nagoya University on November, 2015. The author expresses his sincere gratitude to Professor Ken Kamano for giving him the manuscript and helpful comments on the proof of the main theorem of \cite{K}. The author would also like to thank his advisor Professor Kenichi Bannai for his continuous encouragement, and Dr. Shuji Yamamoto for useful discussions and helpful advice about Theorem \ref{main theorem}. The author is also grateful to the members of the Department of Mathematics at Keio University for their hospitality.

\end{document}